\documentclass[11pt]{amsart}

\usepackage{amsmath,amsthm}
\usepackage{amsfonts}
\usepackage[nobysame]{amsrefs}
\usepackage{verbatim}
\usepackage{amssymb}
\usepackage{amscd}
\usepackage[colorlinks=true, citecolor=blue, 
]{hyperref}
\usepackage{mathrsfs}

\usepackage{tikz}
\usetikzlibrary{matrix}

\usepackage{fouriernc}

\newtheorem{theorem}{Theorem}[section]
\newtheorem{corollary}[theorem]{Corollary}
\newtheorem{lemma}[theorem]{Lemma}
\newtheorem{proposition}[theorem]{Proposition}
\newtheorem{definition}[theorem]{Definition}
\newtheorem{question}[theorem]{Question}

\theoremstyle{remark}
\newtheorem{remark}[theorem]{Remark}

\newcommand{\set}[1]{\,\left\{#1\right\}}
\newcommand{\setd}[2]{\,\left\{#1\ \colon\ #2\right\}}

\newcommand{\supp}{\operatorname{supp}}

\newcommand{\ZZ}{\mathbb{Z}}
\newcommand{\CC}{\mathbb{C}}
\newcommand{\NN}{\mathbb{N}}

\newcommand{\DD}{\mathbb{D}}

\newcommand{\opi}{\overline{\pi}}
\newcommand{\prop}{\operatorname{prop}}

\AtEndDocument{\bigskip{\footnotesize%
  \textsc{Mathematical Institute, Leiden University, P.O. Box 9512, 2300 RA Leiden, The Netherlands} \par  
  \textit{E-mail address}: \texttt{y.c.chung@math.leidenuniv.nl} \par
  \addvspace{\medskipamount}
  \textsc{Institute of Mathematics of the Polish Academy of Sciences, \'{S}niadeckich 8, 00-656 Warsaw, Poland} \par  
  \textit{E-mail address}: \texttt{pnowak@impan.pl} \par
}}

\begin{document}
\title{Expanders are counterexamples to the $\ell^p$ coarse Baum-Connes conjecture}
\author{Yeong Chyuan Chung \and Piotr W. Nowak}

\date{\today}
\maketitle

\begin{abstract}
We consider an $\ell^p$ coarse Baum-Connes assembly map for $1<p<\infty$, and show that it is not surjective for expanders arising from residually finite hyperbolic groups.
\end{abstract}

\tableofcontents

\section{Introduction}

The coarse Baum-Connes conjecture, first formulated in \cite{Roe}, is a coarse geometric analog of the original Baum-Connes conjecture for groups \cite{BCH}, and it posits that a certain coarse assembly map or index map is an isomorphism. On one side of the map is a topological object involving the $K$-homology of Rips complexes of a bounded geometry metric space, while on the other side is the $K$-theory of a certain $C^\ast$-algebra associated to the metric space, namely the Roe algebra, which encodes the large scale geometry of the space. 
The significance of this conjecture lies in its applications in geometry and topology, which includes the Novikov conjecture when the metric space is a finitely generated group equipped with the word metric, and also the problem of existence of positive scalar curvature metrics, as well as Gromov's zero-in-the-spectrum conjecture, when the space is a Riemannian manifold. In fact, injectivity of the coarse assembly map (commonly referred to as the coarse Novikov conjecture) is sufficient for some of these applications.

The conjecture has been proven in a number of cases. Yu showed that it holds for metric spaces that coarsely embed into Hilbert space \cite{Yu00}, generalizing his earlier work showing that it holds for spaces with finite asymptotic dimension \cite{Yu98}. In \cite{Yu98}, there is also a counterexample to the coarse Novikov conjecture when the condition of bounded geometry is omitted. More recent positive results on the coarse Baum-Connes conjecture include \cite{FO12,FO15,FO16,FO18} by Fukaya and Oguni.

On the other hand, Higson \cite{Hig99} showed that the coarse assembly map is not surjective for certain Margulis-type expanders. Then in \cite{HLS}, Higson-Lafforgue-Skandalis showed that for any expander, either the coarse assembly map fails to be surjective or the Baum-Connes assembly map with certain coefficients for an associated groupoid fails to be injective, and that the former always occurs for certain Margulis-type expanders.
Although expanders provide counterexamples to surjectivity of the coarse assembly map, it is known that the map (or the version with maximal Roe algebras) is injective for certain classes of expanders \cite{ONL,GWY,GTY,OYmax}.

While most of the results mentioned in the previous paragraph only apply to Margulis-type expanders, Willett and Yu in \cite{WillettYuI} considered spaces of graphs with large girth and showed that the coarse assembly map is injective for such spaces while it is not surjective if the space is a weak expander. For the maximal version of the coarse assembly map, they showed in \cite{WillettYuII} that it is an isomorphism for such spaces if there is a uniform bound on the vertex degrees of the graphs. They also discussed how their methods can be modified to yield the same results for a version of the coarse assembly map for uniform Roe algebras formulated by \v{S}pakula \cite{Spakula}.

In this paper, our goal is to consider an $\ell^p$ analog of the coarse Baum-Connes assembly map for $1<p<\infty$, and show that it fails to be surjective for certain expanders by adapting the arguments in \cite{WillettYuI}. It should be noted that techniques used in the $C^\ast$-algebraic setting often do not transfer to the $\ell^p$ setting in a straightforward manner. 

Although the $\ell^p$ analog of the coarse Baum-Connes conjecture has no known geometric or topological applications when $p\neq 2$, in light of interest in $L^p$ (or $\ell^p$) operator algebras in recent years (e.g. \cite{Chung2,ChungLiRigidity,ChungLiLpunifRoe,Engel,Gardella,GardLupGrpoid,GardThielGrp,GardThielConv,Pooya,LY,Phil12,Phil13,PhilBle,PhilViola}), the study of assembly maps involving $L^p$ operator algebras contributes to our general understanding of the $K$-theory of some of these algebras. 
We also note that other assembly maps involving $L^p$ operator algebras have recently been considered in \cite{Chung2,Engel}, as well as in unpublished work of Kasparov-Yu. In similar spirit,  the Bost conjecture \cite{Laf02,Par09,Par13,Skan} asks whether the Baum-Connes-type assembly map into the $K$-theory
of the Banach algebra $L^1(G)$ is an isomorphism for a locally compact group $G$.

The description of the assembly map that we use is equivalent to that in \cite{WillettYuI} when $p=2$ in the sense that one is an isomorphism if and only if the other is. This equivalence was established in \cite{LocAlgBC}, where Yu introduced localization algebras and showed that a local index map from $K$-homology to the $K$-theory of the localization algebra is an isomorphism for finite-dimensional simplicial complexes. Qiao and Roe \cite{QR} later showed that this isomorphism holds for general locally compact metric spaces. By considering analogs of Roe algebras and localization algebras represented on appropriate $\ell^p$ spaces, we obtain an $\ell^p$ analog of the coarse Baum-Connes assembly map of the form
\[ \mu:\lim_{R\rightarrow\infty}K_*(B^p_L(P_R(X)))\rightarrow K_*(B^p(X)), \]
and the $\ell^p$ coarse Baum-Connes conjecture is the statement that this map is an isomorphism. 
Following the ideas in \cite{QR} and \cite[Appendix B]{GuWY}, we also provide an alternative form of the assembly map (see Theorem \ref{ThmAltAssembly}).

Zhang and Zhou showed in \cite[Proposition 5.20]{ZhangZhou} that the left-hand side of the $\ell^p$ coarse Baum-Connes assembly map can be naturally identified with that of the original coarse Baum-Connes assembly map. Therefore, given a metric space $X$ with bounded geometry, if the $\ell^p$ coarse Baum-Connes conjecture for $X$ holds for a range of values of $p$, then the $K$-theory groups $K_*(B^p(X))$ are independent of $p$ in that range. This is true when $X$ has finite asymptotic dimension and $p\in(1,\infty)$ by \cite[Theorem 4.6]{ZhangZhou}.
Also, Shan and Wang showed in \cite{ShanWang} that the $\ell^p$ coarse Baum-Connes assembly map is injective for $p\in(1,\infty)$ when $X$ coarsely embeds into a simply connected complete Riemannian manifold of non-positive sectional curvature.

We end the introduction with the statement of our main theorem: 

\begin{theorem}
Let $p\in(1,\infty)$. Let $G$ be a residually finite hyperbolic group. Let $N_1\supseteq N_2\supseteq\cdots$ be a sequence of normal subgroups of finite index such that $\bigcap_iN_i=\{e\}$. Assume that the box space $\square G=\bigsqcup_i G/N_i$ is an expander, i.e., that $G$ has property $\tau$ with respect to the family $\{N_i\}$. If $q\in B^p(\square G)$ is the Kazhdan projection associated to $\square G$, then $[q]\in K_0(B^p(\square G))$ is not in the image of the $\ell^p$ coarse Baum-Connes assembly map. 
\end{theorem}


As examples, our result applies to finitely generated free groups and $SL_2(\ZZ)$.

One can see from the proof of our main result that the hyperbolicity assumption can in fact be replaced by the following set of conditions:
\begin{enumerate}
\item $G$ has the $p$-operator norm localization property.
\item The $\ell^p$ Baum-Connes assembly map for each $N_i$ is injective.
\item The classifying space for proper $G$-actions has finite homotopy type, i.e., there is a model $Z$ of a locally finite CW complex with universal proper $G$-action such that $Z/G$ is a compact CW complex.
\end{enumerate}

\subsection*{Acknowledgements} 
We would like to thank Rufus Willett for his illuminating comments.
We also thank the referee for carefully reading our manuscript and providing useful suggestions to improve our exposition.

This project has received funding from the European Research Council (ERC) under the European Union's Horizon 2020 
research and innovation programme (grant agreement no. 677120-INDEX).

\section{The $\ell^p$ coarse Baum-Connes assembly map}

In this section, we define the $\ell^p$ coarse Baum-Connes assembly map and also consider the equivariant version of the assembly map.

\subsection{Roe algebras and localization algebras}
We begin by introducing the $\ell^p$ Roe algebra and $\ell^p$ localization algebra, whose $K$-theory groups are used to define the $\ell^p$ coarse Baum-Connes assembly map.

\begin{definition}
Let $(X,d)$ be a metric space. 
\begin{enumerate}
\item We say that $X$ is uniformly discrete if there exists $\delta>0$ such that $d(x,y)\geq\delta$ for all distinct points $x,y\in X$.
\item If $X$ is uniformly discrete, we say that it has bounded geometry if for all $R>0$ there exists $N_R\in\NN$ such that all balls of radius $R$ in $X$ have cardinality at most $N_R$.
\item We say that $X$ is proper if all closed balls in $X$ are compact.
\item A net in $X$ is a discrete subset $Y\subseteq X$ such that there exists $r>0$ with the properties that $d(x,y)\geq r$ for all $x,y\in Y$, and for any $x\in X$ there is $y\in Y$ with $d(x,y)<r$.
\item If $X$ is proper, we say that it has bounded geometry if it contains a net with bounded geometry.
\end{enumerate}
\end{definition}

We now associate certain $\ell^p$ operator algebras called Roe algebras to proper metric spaces. These algebras encode the large scale geometry of the metric space, and the $K$-theory of the Roe algebra serves as the target of the $\ell^p$ coarse Baum-Connes assembly map. 

\begin{definition} \label{DefRoe}
Let $X$ be a proper metric space, and fix a countable dense subset $Z\subseteq X$. Let $T$ be a bounded operator on $\ell^p(Z,\ell^p)$, and write $T=(T_{x,y})_{x,y\in Z}$ so that each $T_{x,y}$ is a bounded operator on $\ell^p$. $T$ is said to be locally compact if
\begin{itemize}
\item each $T_{x,y}$ is a compact operator on $\ell^p$;
\item for every bounded subset $B\subseteq X$, the set \[\setd{(x,y)\in (B\times B)\cap(Z\times Z)}{T_{x,y}\neq 0}\]is finite.
\end{itemize}
The propagation of $T$ is defined to be
\[ \prop(T)=\inf\setd{S>0}{T_{x,y}=0\;\text{for all}\;x,y\in Z\;\text{with}\;d(x,y)>S}. \]
The algebraic Roe algebra of $X$, denoted $\CC^p[X]$, is the subalgebra of $B(\ell^p(Z,\ell^p))$ consisting of all finite propagation, locally compact operators. The $\ell^p$ Roe algebra of $X$, denoted $B^p(X)$, is the closure of $\CC^p[X]$ in $B(\ell^p(Z,\ell^p))$. 

If $X$ is uniformly discrete, we define $\CC^p_u[X]$ to be the subalgebra of $B(\ell^p(X))$ consisting of all finite propagation operators, and we define the $\ell^p$ uniform Roe algebra of $X$, denoted $B^p_u(X)$, to be the closure of $\CC^p_u[X]$ in $B(\ell^p(X))$. 
\end{definition}

One can check that just like in the $p=2$ case, up to non-canonical isomorphism, $B^p(X)$ does not depend on the choice of dense subspace $Z$, while up to canonical isomorphism, its $K$-theory does not depend on the choice of $Z$. 
These follow from \cite[Corollary 2.9]{ZhangZhou} (and the results preceding it).
Moreover, $B^p(X)$ is a coarse invariant, as noted in \cite{ChungLiRigidity}.

\begin{remark} 
In the definition above, one may consider bounded operators on $\ell^p(Z,E)$ for a fixed separable infinite-dimensional $L^p$ space $E$.
We refer the reader to \cite[Subsection 4.9.3]{Pietsch} and the references therein for the following classification of the separable infinite-dimensional $L^p$ spaces when $p\in[1,\infty)\setminus\{2\}$:
\begin{itemize}
\item Up to isometric isomorphism: $\ell^p$, $L^p[0,1]$, $L^p[0,1]\oplus_p\ell^p_n$, $L^p[0,1]\oplus_p\ell^p$, where $\ell^p_n$ denotes $\mathbb{C}^n$ with the $\ell^p$ norm.
\item Up to non-isometric isomorphism: $\ell^p$, $L^p[0,1]$.
\end{itemize}
Different choices of $E$ may result in non-isomorphic Banach algebras. For example, if $X$ is a point, then using $\ell^p$ in the definition results in $B^p(X)=K(\ell^p)$, while using $L^p[0,1]$ in the definition results in $B^p(X)=K(L^p[0,1])$, and these two Banach algebras are non-isomorphic.
\end{remark}

The domain of the $\ell^p$ coarse Baum-Connes assembly map involves the $K$-theory of localization algebras of Rips complexes of the metric space. We will now define these notions and formulate the $\ell^p$ coarse Baum-Connes conjecture.

\begin{definition} \label{DefLoc}
Let $X$ be a proper metric space, and let $\CC^p[X]$ be its algebraic Roe algebra. Let $\CC^p_L[X]$ be the algebra of bounded, uniformly continuous functions $f:[0,\infty)\rightarrow\CC^p[X]$ such that $\prop(f(t))\rightarrow 0$ as $t\rightarrow\infty$. Equip $\CC^p_L[X]$ with the norm
\[ ||f||:=\sup_{t\in[0,\infty)}||f(t)||_{B^p(X)}. \]
The completion of $\CC^p_L[X]$ under this norm, denoted by $B^p_L(X)$, is the $\ell^p$ localization algebra of $X$.
\end{definition}

\begin{definition}
Let $(X,d)$ be a bounded geometry metric space, and let $R>0$. The Rips complex of $X$ at scale $R$, denoted $P_R(X)$, is the simplicial complex with vertex set $X$ and such that a finite set $\set{x_1,\ldots,x_n}\subseteq X$ spans a simplex if and only if $d(x_i,x_j)\leq R$ for all $i,j=1,\ldots,n$.

Equip $P_R(X)$ with the spherical metric defined by identifying each $n$-simplex with the part of the $n$-sphere in the positive orthant, and equipping $P_R(X)$ with the associated length metric.
\end{definition}

For any $R>0$, there is a homomorphism 
\[ i_R:K_*(B^p(P_R(X)))\rightarrow K_*(B^p(X)), \] 
(cf. \cite[Lemma 2.8]{LocAlgBC} for the $p=2$ case)
and the $\ell^p$ coarse Baum-Connes assembly map
\[ \mu:\lim_{R\rightarrow\infty}K_*(B^p_L(P_R(X)))\rightarrow K_*(B^p(X)) \]
is defined to be the limit of the composition
\[ K_*(B^p_L(P_R(X)))\stackrel{e_*}{\rightarrow} K_*(B^p(P_R(X)))\stackrel{i_R}{\rightarrow} K_*(B^p(X)), \]
where $e:B^p_L(P_R(X))\rightarrow B^p(P_R(X))$ is the evaluation-at-zero map.

The $\ell^p$ coarse Baum-Connes conjecture for a proper metric space $X$ with bounded geometry is the statement that $\mu$ is an isomorphism.

It was shown in \cite[Proposition 5.20]{ZhangZhou} that the left-hand side of the $\ell^p$ coarse Baum-Connes assembly map can be naturally identified with that of the original coarse Baum-Connes assembly map.

\subsection{Equivariant assembly maps}

We will also require equivariant versions of the Roe algebra, localization algebra, and assembly map. This subsection contains all the necessary information about these.

\begin{definition} \label{DefEquivRoe}
Let $X$ be a proper metric space, and let $\Gamma$ be a countable discrete group acting freely and properly on $X$ by isometries. Fix a $\Gamma$-invariant countable dense subset $Z\subseteq X$ and define $\CC^p[X]$ as above. The equivariant algebraic Roe algebra of $X$, denoted $\CC^p[X]^\Gamma$, is the subalgebra of $\CC^p[X]$ consisting of matrices $(T_{x,y})_{x,y\in Z}$ satisfying $T_{gx,gy}=T_{x,y}$ for all $g\in\Gamma$ and $x,y\in Z$. The equivariant $\ell^p$ Roe algebra of $X$, denoted $B^p(X)^\Gamma$, is the closure of $\CC^p[X]^\Gamma$ in $B(\ell^p(Z,\ell^p))$.

The equivariant $\ell^p$ localization algebra $B^p_L(X)^\Gamma$ is defined by considering the algebra of bounded, uniformly continuous functions $f:[0,\infty)\rightarrow\CC^p[X]^\Gamma$ such that $\prop(f(t))\rightarrow 0$ as $t\rightarrow\infty$, and completing in the norm
\[ ||f||:=\sup_{t\in[0,\infty)}||f(t)||_{B^p(X)^\Gamma}. \]
\end{definition}

If $\Gamma$ is a discrete group, then we may represent the group ring $\CC\Gamma$ on $\ell^p(\Gamma)$ by left translation. Its completion, which we denote by $B^p_r(\Gamma)$, is the reduced $L^p$ group algebra of $\Gamma$, also known as the algebra of $p$-pseudofunctions on $\Gamma$ in the literature. When $p=2$, it is the reduced group $C^\ast$-algebra.

Just as in the $p=2$ case, the equivariant $\ell^p$ Roe algebra of $X$ as defined above is related to the reduced $L^p$ group algebra of $\Gamma$. Before making this precise, we recall some facts about $L^p$ tensor products, details of which can be found in \cite[Chapter 7]{DF}.

For $p\in[1,\infty)$, there is a tensor product of $L^p$ spaces such that we have a canonical isometric isomorphism $L^p(X,\mu)\otimes L^p(Y,\nu)\cong L^p(X\times Y,\mu\times\nu)$, which identifies, for every $\xi\in L^p(X,\mu)$ and $\eta\in L^p(Y,\nu)$, the element $\xi\otimes\eta$ with the function $(x,y)\mapsto\xi(x)\eta(y)$ on $X\times Y$. Moreover, we have the following properties:
\begin{itemize}
\item Under the identification above, the linear span of all $\xi\otimes\eta$ is dense in $L^p(X\times Y,\mu\times\nu)$.
\item $||\xi\otimes\eta||_p=||\xi||_p||\eta||_p$ for all $\xi\in L^p(X,\mu)$ and $\eta\in L^p(Y,\nu)$.
\item The tensor product is commutative and associative.
\item If $a\in B(L^p(X_1,\mu_1),L^p(X_2,\mu_2))$ and $b\in B(L^p(Y_1,\nu_1),L^p(Y_2,\nu_2))$, then there exists a unique $c\in B(L^p(X_1\times Y_1,\mu_1\times\nu_1),L^p(X_2\times Y_2,\mu_2\times\nu_2))$ such that under the identification above, $c(\xi\otimes\eta)=a(\xi)\otimes b(\eta)$ for all $\xi\in L^p(X_1,\mu_1)$ and $\eta\in L^p(Y_1,\nu_1)$. We will denote this operator by $a\otimes b$. Moreover, $||a\otimes b||=||a|| ||b||$.
\item The tensor product of operators is associative, bilinear, and satisfies $(a_1\otimes b_1)(a_2\otimes b_2)=a_1a_2\otimes b_1b_2$.
\end{itemize}
If $A\subseteq B(L^p(X,\mu))$ and $B\subseteq B(L^p(Y,\nu))$ are norm-closed subalgebras, we then define $A\otimes B\subseteq B(L^p(X\times Y,\mu\times\nu))$ to be the closed linear span of all $a\otimes b$ with $a\in A$ and $b\in B$.

Regarding $M_n(\CC)$ as $B(\ell^p_n)$, we may view $M_n(A)$ as $M_n(\CC)\otimes A$ when $A$ is an $L^p$ operator algebra and the tensor product is as described above (see Remark 1.14 and Example 1.15 in \cite{Phil13}). Writing $\overline{M^p_\infty}$ for $\overline{\bigcup_{n\in\NN}M_n(\CC)}^{B(\ell^p)}$, we see that $\overline{M^p_\infty}$ is a closed subalgebra of $B(\ell^p)$. Let $P_n$ be the projection onto the first $n$ coordinates with respect to the standard basis in $\ell^p$. When $p\in(1,\infty)$, we have $\lim_{n\rightarrow\infty}||a-P_n aP_n||=0$ for any compact operator $a\in K(\ell^p)$. It follows that $\overline{M^p_\infty}=K(\ell^p)$ for $p\in(1,\infty)$. However, when $p=1$, we can only say that $\lim_{n\rightarrow\infty}||a-P_n a||=0$ for $a\in K(\ell^1)$. In fact, there is a rank one operator on $\ell^1$ that is not in $\overline{M^1_\infty}$. We refer the reader to Proposition 1.8 and Example 1.10 in \cite{Phil13} for details.

The standard proof in the $p=2$ case allows one to show that if $A$ is an $L^p$ operator algebra for some $p\in[1,\infty)$, then
\[ K_*(\overline{M^p_\infty}\otimes A)\cong K_*(A). \]
In particular, when $p\in(1,\infty)$, we have \[ K_*(K(\ell^p)\otimes A)\cong K_*(A). \]
We refer to \cite[Lemma 6.6]{Phil13} for details.

The following lemma is well-known when $p=2$ (cf. \cite[Lemma 3.7]{WillettYuI}).

\begin{lemma} \label{equivariant}
Let $\Gamma$ be a discrete group acting freely, properly, and cocompactly by isometries on a proper metric space $X$. Let $Z\subseteq X$ be the countable dense $\Gamma$-invariant subset used to define $\CC[X]^\Gamma$. Let $D\subseteq Z$ be a precompact fundamental domain for the $\Gamma$-action on $Z$. 
Then for $p\in(1,\infty)$ there is an isomorphism
\[ \psi_D:B^p(X)^\Gamma\rightarrow B^p_r(\Gamma)\otimes K(\ell^p(D,\ell^p)). \]
Moreover, the induced isomorphism on $K$-theory is independent of the choice of $D$.
\end{lemma}

\begin{proof}
Let $K_f(\ell^p(D,\ell^p))$ be the dense subalgebra of $K(\ell^p(D,\ell^p))$ consisting of those operators $(K_{x,y})_{x,y\in D}$ with only finitely many nonzero matrix entries.
For $T\in\CC^p[X]^\Gamma$ and $g\in\Gamma$, define an element $T^{(g)}\in K_f(\ell^p(D,\ell^p))$ by the matrix formula
\[ T^{(g)}_{x,y}:=T_{x,gy}\;\text{for all}\;x,y\in D. \]
Define a homomorphism
\[ \psi_D:\CC^p[X]^\Gamma\rightarrow\CC\Gamma\odot K_f(\ell^p(D,\ell^p)) \]
by the formula
\[ T\mapsto\sum_{g\in\Gamma}\lambda_g\odot T^{(g)}. \]
Note that only finitely many $T^{(g)}$ are nonzero since $T$ has finite propagation. Moreover, $\psi_D$ is an isomorphism.

In fact, $\psi_D$ is implemented by conjugating $T$ by the isometric isomorphism 
\[U:\ell^p(Z,\ell^p)\rightarrow\ell^p(\Gamma)\otimes\ell^p(D,\ell^p), \quad\xi\mapsto \sum_{g\in\Gamma}\delta_g\otimes \chi_D U_g\xi,\] 
i.e., $\psi_D(T)=UTU^{-1}$, and so $\psi_D$ extends to an isometric isomorphism between the completions.

If $D'\subseteq Z$ is another precompact fundamental domain for the $\Gamma$-action, then $\psi_D(T)$ and $\psi_{D'}(T)$ differ by conjugation by an invertible multiplier of $B^p_r(\Gamma)\otimes K(\ell^p(D',\ell^p))$ (i.e., an element $a\in B(\ell^p(Z,\ell^p))$ such that $B^p_r(\Gamma)\otimes K(\ell^p(D',\ell^p))$ is closed under left multiplication and right multiplication by $a$), which induces the identity map on $K$-theory (cf. \cite[Lemma 4.6.1]{HR}).
\end{proof}

If $\Gamma$ is a countable discrete group acting freely and properly on $X$ by isometries, then by considering the equivariant versions of the localization algebra and the Roe algebra, we have the $\ell^p$ equivariant assembly map
\[ \lim_{R\rightarrow\infty}K_*(B^p_L(P_R(X))^\Gamma)\rightarrow K_*(B^p(X)^\Gamma)\cong K_*(B^p_r(\Gamma)), \]
which is (a model for) the Baum-Connes assembly map for $\Gamma$ when $p=2$ \cite{Shan}.

As discussed in \cite[Section 6.2]{Chung2}, the domain of the $\ell^p$ equivariant assembly map can be identified with that of the classical Baum-Connes assembly map, similar to how it is done for the coarse assembly map in \cite[Section 5]{ZhangZhou}.

Using involutive versions of the $\ell^p$ Roe algebra and localization algebra, we can further establish a relationship between the $\ell^p$ equivariant assembly map and the classical Baum-Connes assembly map for a given group.
We shall use some terminology from \cite[Section 5]{ZhangZhou} for this purpose.

Given $p\in(1,\infty)$, and countable discrete measure spaces $Z$ and $Z'$, the Banach spaces $\ell^p(Z)$ and $\ell^p(Z')$ have canonical Schauder bases $(e_i)_{i\in Z}$ and $(e_i')_{i\in Z'}$ respectively. Any bounded linear operator $T\in B(\ell^p(Z),\ell^p(Z'))$ may be viewed as an infinite matrix with respect to the Schauder bases. We may then consider the conjugate transpose matrix $T^*$. 

We say that $T\in B(\ell^p(Z),\ell^p(Z'))$ is a dual-operator if $T^*$ is a bounded operator from $\ell^p(Z')$ to $\ell^p(Z)$.
In this case, $T$ may be regarded as a bounded linear operator from $\ell^q(Z)$ to $\ell^q(Z')$ with $1/p+1/q=1$.
We say that $T$ is a compact dual-operator if $T$ and $T^*$ are compact operators from $\ell^p(Z)$ to $\ell^p(Z')$, and from $\ell^p(Z')$ to $\ell^p(Z)$, respectively.
Given a dual-operator $T$, we define its maximal norm by $||T||_{max}=\max(||T||_{B(\ell^p(Z),\ell^p(Z'))},||T^*||_{B(\ell^p(Z'),\ell^p(Z))})$.

One can now consider the following definition analogous to Definition \ref{DefRoe} and Definition \ref{DefEquivRoe}.

\begin{definition}
Let $X$ be a proper metric space, and fix a countable dense subset $Z\subseteq X$. A dual-operator $T=(T_{x,y})_{x,y\in Z}\in B(\ell^p(Z,\ell^p))$ is said to be locally compact if
\begin{itemize}
\item each $T_{x,y}$ is a compact dual-operator on $\ell^p$;
\item for every bounded subset $B\subseteq X$, the set \[\{(x,y)\in (B\times B)\cap(Z\times Z):T_{x,y}\neq 0\}\] is finite.
\end{itemize}
The dual $\ell^p$ Roe algebra of $X$, denoted by $B^{p,*}(X)$, is the maximal norm closure of the algebra of all locally compact dual-operators on $\ell^p(Z,\ell^p)$ with finite propagation.

Let $\Gamma$ be a countable discrete group acting freely and properly on $X$ by isometries. 
The equivariant dual $\ell^p$ Roe algebra, denoted by $B^{p,*}(X)^\Gamma$, is the maximal norm closure of the algebra of all locally compact dual-operators on $\ell^p(Z,\ell^p)$ with finite propagation and satisfying $T_{gx,gy}=T_{x,y}$ for all $g\in\Gamma$ and $x,y\in Z$.
\end{definition}

The completion of the group ring $\CC\Gamma$ in the maximal norm \[||f||_{max}=\max(||f||_{B(\ell^p(\Gamma))},||f^*||_{B(\ell^p(\Gamma))}),\] where $f^*(g)=\overline{f(g^{-1})}$, is the involutive version of the reduced $L^p$ group algebra of $\Gamma$, and is denoted by $B^{p,*}(\Gamma)$. One obtains the involutive version of Lemma \ref{equivariant} with essentially the same proof.

The dual $\ell^p$ localization algebra $B^{p,*}_L(X)$ and the equivariant dual $\ell^p$ localization algebra $B^{p,*}_L(X)^\Gamma$ are defined in the obvious manner analogous to Definition \ref{DefLoc} and Definition \ref{DefEquivRoe}.

One immediately sees that there are inclusion homomorphisms
\begin{align*}
B^{p,*}(X)^\Gamma &\rightarrow B^p(X)^\Gamma, \\
B^{p,*}_L(X)^\Gamma &\rightarrow B^p_L(X)^\Gamma.
\end{align*}
Moreover, using the Riesz-Thorin interpolation theorem, one obtains contractive homomorphisms
\begin{align*}
B^{p,*}(X)^\Gamma &\rightarrow C^*(X)^\Gamma=B^2(X)^\Gamma, \\
B^{p,*}_L(X)^\Gamma &\rightarrow C^*_L(X)^\Gamma=B^2_L(X)^\Gamma.
\end{align*}
When $X$ is a finite-dimensional simplicial complex, one can show that the homomorphisms between the localization algebras induce isomorphisms on $K$-theory as outlined in \cite[Section 6.2]{Chung2} (cf. \cite[Propositions 5.18 and 5.19]{ZhangZhou} for the non-equivariant case).

\begin{proposition}
If the Baum-Connes assembly map for $\Gamma$ is injective, and the inclusion $B^{p,*}_r(\Gamma)\rightarrow B^p_r(\Gamma)$ induces an injection on $K$-theory, then the $\ell^p$ Baum-Connes assembly map for $\Gamma$ is injective.
\end{proposition}

\begin{proof}
The result follows immediately from the following commutative diagram of assembly maps:
\[
\begin{CD}
\lim_{R\rightarrow\infty}K_*(B^p_L(P_R(X))^\Gamma)	@>e_*>>	K_*(B^p(X)^\Gamma) \cong K_*(B^p_r(\Gamma))		 \\
@A\cong AA  @AAA \\
\lim_{R\rightarrow\infty}K_*(B^{p,*}_L(P_R(X))^\Gamma)	@>e_*>>	K_*(B^{p,*}(X)^\Gamma) \cong K_*(B^{p,*}_r(\Gamma))	 \\
@V\cong VV		@VVV \\
\lim_{R\rightarrow\infty}K_*(C^*_L(P_R(X))^\Gamma)	@>e_*>>	K_*(C^*(X)^\Gamma) \cong K_*(C^*_r(\Gamma))
\end{CD}
\]
The bottom horizontal arrow is the Baum-Connes assembly map for $\Gamma$.
The top horizontal arrow is the $\ell^p$ Baum-Connes assembly map for $\Gamma$.
The upper vertical arrow on the right is induced by inclusion, and the lower vertical arrow on the right is induced by complex interpolation.
\end{proof}

In \cite[Definition 4.1]{LY}, Liao and Yu introduced a Banach version $(RD)_q$ of the rapid decay property for groups, and considered the question of when the inclusion $B^{p,*}_r(\Gamma)\rightarrow B^p_r(\Gamma)$ induces an isomorphism on $K$-theory.

\begin{proposition} \cite[Corollary 4.9]{LY}
Let $q_0\in[1,\infty]$, $q\in[1,q_0]$, and $1/p+1/q=1$. If $\Gamma$ has property $(RD)_{q_0}$, then the inclusion $B^{p,*}_r(\Gamma)\rightarrow B^p_r(\Gamma)$ induces an isomorphism on $K$-theory.
\end{proposition}

By \cite[Theorem 4.4]{LY}, the result applies to all groups with property RD.
Moreover, groups with property RD in Lafforgue's class $\mathcal{C}'$ satisfy the Baum-Connes conjecture \cite[Corollaire 0.0.4]{Laf02}. 

\begin{corollary} \label{LpBC}
The $\ell^p$ Baum-Connes assembly map is injective for groups with property RD in Lafforgue's class $\mathcal{C}'$.
\end{corollary}

Lafforgue's class $\mathcal{C}'$ includes groups acting properly and isometrically on a strongly bolic, weakly geodesic, uniformly locally finite metric space, which by \cite{Mineyev-Yu} includes hyperbolic groups (and their subgroups).
Hyperbolic groups also have property RD \cite{de la Harpe}.
Thus the $\ell^p$ Baum-Connes assembly map is injective for hyperbolic groups.

Finally, we record an induction isomorphism from the $K$-theory of the localization algebra of a metric space to the $K$-theory of the equivariant localization algebra of a covering space.
For the proof of this result, we have taken the approach in the proof of \cite[Lemma 12.5.4]{HR} together with material from our appendix but we note that the proof of \cite[Theorem 6.5.15]{WY} can also be adapted to obtain the result. 

\begin{definition}
Let $X$ and $\tilde{X}$ be metric spaces, and let $\pi:\tilde{X}\rightarrow X$ be a surjective map. Let $\varepsilon>0$. Then $(\tilde{X},\pi)$ is called an $\varepsilon$-metric cover of $X$ if for all $x\in\tilde{X}$, the restriction of $\pi$ to the open ball $B(x,\varepsilon)$ of radius $\varepsilon$ around $x$ in $\tilde{X}$ is an isometry onto the open ball $B(\pi(x),\varepsilon)$ of radius $\varepsilon$ around $\pi(x)$ in $X$.
\end{definition}

For example, if $X$ is a length space (see \cite[Definition I.3.1]{BH}), and the metric on $\tilde{X}$ is the one canonically induced by $\pi$, then $(\tilde{X},\pi)$ is an $\varepsilon$-metric cover for some $\varepsilon>0$ \cite[Proposition I.3.25]{BH}.

\begin{proposition} \label{ind}
Let $\pi:\tilde{Y}\rightarrow Y$ be a regular (or Galois) cover with deck transformation group $\Gamma$, and suppose $(\tilde{Y},\pi)$ is an $\varepsilon$-metric cover for some $\varepsilon>0$.
Then we have an induction isomorphism \[ind:K_*(B^p_L(Y))\stackrel{\cong}{\rightarrow} K_*(B^p_L(\tilde{Y})^\Gamma).\]
\end{proposition}

\begin{proof}
We prove that there is an isomorphism \[D^p_L(Y)/B^p_L(Y) \stackrel{\cong}{\rightarrow} D^p_L(\tilde{Y})^\Gamma/B^p_L(\tilde{Y})^\Gamma,\]
referring the reader to Definition \ref{def:dpl} for the definition of $D^p_L(Y)$.
By Proposition \ref{PropDLtrivialK} and its equivariant analog, we then get an isomorphism \[K_*(B^p_L(Y))\stackrel{\cong}{\rightarrow} K_*(B^p_L(\tilde{Y})^\Gamma).\]

Let $\varepsilon>0$ be such that $(\tilde{Y},\pi)$ is an $\varepsilon$-metric cover. 
Let $T$ be a bounded linear operator on $\ell^p(Z,\ell^p)$ that has propagation less than $\varepsilon/2$, where $Z$ is a countable dense subset of $Y$ containing a finite net $\{z_i\}_{i=1}^N$ in $Y$ with bounded geometry.
Consider a cover $\mathcal{U}=\set{U_i}_{i=1}^N$ of $Y$ by uniformly bounded, pairwise disjoint Borel sets of diameter less than $\varepsilon/2$ such that $z_i\in U_i$ for each $i\in I$, and let $\chi_i$ be the characteristic function of $U_i$. Define 
$$T_{i,j} = \chi_i T\chi_j: \chi_j \ell^p(Z,\ell^p)=\ell^p(U_j\cap Z,\ell^p) \to \ell^p(U_i\cap Z,\ell^p)=\chi_i \ell^p(Z,\ell^p).$$
The cover $\mathcal{U}$ can be lifted to a cover $\widetilde{\mathcal{U}}$ of the covering space $\widetilde{Y}$ by setting $\widetilde{U}_i=\pi^{-1}(U_i)$.
Then $\widetilde{U}_i=\bigsqcup_{g\in G} U_{i,g} \cong U_i\times \Gamma$, where $\pi$ restricted to each $U_{i,g}$ is an isometry onto $U_i$.

This allows us to lift the operator $T_{i,j}$ to the operator
$$\widetilde{T_{i,j}} :\chi_{j,h}(\ell^p(\Gamma)\otimes\ell^p(Z)\otimes\ell^p) \to \chi_{i,g}(\ell^p(\Gamma)\otimes\ell^p(Z)\otimes\ell^p),$$
where $\chi_{i,g}$ is the characteristic function of $U_{i,g}$, i.e., $\widetilde{T_{i,j}}$ may be identified with $T_{i,j}$ via the following diagram:
\[
\begin{CD}
\chi_{j,h}(\ell^p(\Gamma)\otimes\ell^p(Z)\otimes\ell^p)	@>\widetilde{T_{i,j}}>>	\chi_{i,g}(\ell^p(\Gamma)\otimes\ell^p(Z)\otimes\ell^p) \\
@|  				@|	\\
\ell^p(\{h\})\otimes\ell^p(U_j\cap Z,\ell^p)		@>T_{i,j}>>	\ell^p(\{g\})\otimes\ell^p(U_i\cap Z,\ell^p)
\end{CD}
\]

Let now $\mathcal{L}_{\delta}[Y]$ denote the set of bounded linear operators on $\ell^p(Z,\ell^p)$ of propagation less than $\delta$; 
let also $\mathcal{L}_{\delta}[\widetilde{Y}]^\Gamma$ denote the set of bounded, $\Gamma$-equivariant operators on $\ell^p(\Gamma)\otimes\ell^p(Z)\otimes\ell^p$ of propagation less than $\delta$.

For an operator $T\in \mathcal{L}_{\varepsilon/2}[Y]$, define its lift $\phi_L(T)$ by the formula
$$\phi_L(T)_{(i,g), (j,h)}=\left\lbrace \begin{array}{ll} \widetilde{T_{i,j}} & d(U_{i,g}, U_{j,h}) <\frac{\varepsilon}{2}\\
0 &\text{otherwise}\end{array}\right.$$
Then $\phi_L(T) \in \mathcal{L}_{\varepsilon/2}[\widetilde{Y}]^\Gamma$.
This lifting process gives a one-to-one correspondence between operators in $\mathcal{L}_{\varepsilon/2}[Y]$ and operators in $\mathcal{L}_{\varepsilon/2}[\widetilde{Y}]^\Gamma$. It also preserves the properties of local compactness and pseudolocality.

Note that if $S,T\in\mathcal{L}_{\varepsilon/2}[Y]$ are such that $ST\in\mathcal{L}_{\varepsilon/2}[Y]$, then $\phi_L(ST)=\phi_L(S)\phi_L(T)$. Also, by the definition of $\phi_L$, if $f:[0,\infty)\rightarrow\mathcal{L}_{\varepsilon/2}[Y]$ is a bounded, uniformly continuous function, then so is $\phi_L\circ f:[0,\infty)\rightarrow\mathcal{L}_{\varepsilon/2}[\widetilde{Y}]^\Gamma$, and thus $\phi_L\circ f\in B^p_L(\tilde{Y})^\Gamma$ if $f\in B^p_L(Y)$, and similarly for the $D^p_L$ algebras. 
Indeed, this follows from the fact that since $\widetilde{Y}$ contains a net with bounded geometry $\{z_{i,g}\}_{1\leq i\leq N,g\in\Gamma}$ such that $z_{i,g}\in U_{i,g}$, if $T\in\mathcal{L}_{\delta}[\widetilde{Y}]$, then there exists $c_\delta>0$ such that \[||\phi_L(T)||\leq c_\delta\sup_{(i,g),(j,h)}||\phi_L(T)_{(i,g),(j,h)}||=c_\delta\sup_{i,j}||T_{i,j}||\]
(see \cite[Lemma 2.6]{OYmax} for the $p=2$ case and note that the same reasoning works for all $p$).

By Corollary \ref{CorFinProp}, every element in $D^p_L(Y)/B^p_L(Y)$ can be represented by an element in $D^p_L(Y)$ with arbitrarily small propagation so the lifting procedure above gives the desired isomorphism $D^p_L(Y)/B^p_L(Y) \stackrel{\cong}{\rightarrow} D^p_L(\tilde{Y})^\Gamma/B^p_L(\tilde{Y})^\Gamma$.
\end{proof}

\section{$p$-operator norm localization and a lifting homomorphism}

In this section, we introduce the $p$-operator norm localization property, which was defined for $p=2$ in \cite{ONL}.
It enables us to get a bounded lifting homomorphism on the Roe algebra of the box space of a residually finite group.

\begin{definition}
Let $(X,\nu)$ be a metric space equipped with a positive locally finite Borel measure $\nu$, let $p\in[1,\infty)$, and let $E$ be an infinite-dimensional Banach space. Let $f:\NN\rightarrow\NN$ be a non-decreasing function. We say that $X$ has the $p$-operator norm localization property relative to $f$ (and $E$) with constant $0<c\leq 1$ if for all $r\in\NN$ and every $T\in B(L^p(X,\nu;E))$ with $\prop(T)\leq r$, there exists a nonzero $\xi\in L^p(X,\nu;E)$ satisfying
\begin{enumerate}
\item $\mathrm{diam}(\supp(\xi))\leq f(r)$,
\item $||T\xi||\geq c||T|| ||\xi||$.
\end{enumerate}
\end{definition}

\begin{definition}
A metric space $X$ is said to have the $p$-operator norm localization property if there exists a constant $0<c\leq 1$ and a non-decreasing function $f:\NN\rightarrow\NN$ such that for every positive locally finite Borel measure $\nu$ on $X$, $(X,\nu)$ has the $p$-operator norm localization property relative to $f$ with constant $c$.
\end{definition}

\begin{remark}  (cf. \cite[Proposition 2.4]{ONL})
If a metric space has the $p$-operator norm localization property, then it has the property with constant $c$ for all $0<c<1$.
\end{remark}

It was shown in \cite[Proposition 4.1]{ONL} that the metric sparsification property implies the 2-operator norm localization property. For metric spaces with bounded geometry, it was shown in \cite{ULA} that property A implies the metric sparsification property, and it was shown in \cite{Sako} that property A is equivalent to the 2-operator norm localization property.

As noted in \cite[Section 7]{WillettSpakula}, where a similar property called lower norm localization was considered, the proof of \cite[Proposition 4.1]{ONL} can be adapted with the obvious modifications to yield the following.

\begin{proposition} \label{MSP}
The metric sparsification property implies the $p$-operator norm localization property for every $p\in[1,\infty)$.
\end{proposition}

It is known from \cite{Roe05} that hyperbolic groups have finite asymptotic dimension. Also, finite asymptotic dimension implies the metric sparsification property by \cite[Remark 3.2]{ONL}.

\begin{corollary} \label{pONLhyp}
Hyperbolic groups have the $p$-operator norm localization property for all $p\in[1,\infty)$.
\end{corollary}

Now we consider a lifting map $\phi$ defined in \cite[Lemma 3.8]{WillettYuI} (and also in \cite[Section 7]{ONL}). 

Let $G$ be a finitely generated, residually finite group with a sequence of normal subgroups of finite index $N_1\supseteq N_2\supseteq \cdots$ such that $\bigcap_i N_i=\{e\}$. Let $\square G=\bigsqcup_i G/N_i$ be the box space, i.e., the disjoint union of the finite quotients $G/N_i$, endowed with a metric $d$ such that its restriction to each $G/N_i$ is the quotient metric, while $d(G/N_i,G/N_j)\geq i+j$ if $i\neq j$.

Let $T\in\CC^p[\square G]$ have propagation $S$, and let $M$ be such that for all $i,j\geq M$, we have $d(G/N_i,G/N_j)\geq 2S$ and $\pi_i:G\rightarrow G/N_i$ is a $2S$-metric cover.
We may then write $T=T^{(0)}\oplus\prod_{i\geq M}T^{(i)}$, where $T^{(0)}\in B(\ell^p(G/N_1\sqcup\cdots\sqcup G/N_{M-1},\ell^p))$, and each $T^{(i)}\in B(\ell^p(G/N_i,\ell^p))$.
For each $i\geq M$, define an operator $\widetilde{T^{(i)}}\in\CC^p[G]^{N_i}$ by 
\[ \widetilde{T^{(i)}}_{x,y}=\begin{cases} T^{(i)}_{\pi_i(x),\pi_i(y)} & \text{if}\;d(x,y)\leq S \\ 0 & \text{otherwise} \end{cases}, \]
and define $\phi(T)$ to be the image of $\prod_{i\geq M}\widetilde{T^{(i)}}$ in $\frac{\prod_i \CC^p[G]^{N_i}}{\bigoplus_i \CC^p[G]^{N_i}}$. This defines a homomorphism
\[ \phi:\CC^p[\square G]\rightarrow\frac{\prod_i \CC^p[G]^{N_i}}{\bigoplus_i \CC^p[G]^{N_i}}. \]

\begin{lemma} \label{phihom}
If $G$ has the $p$-operator norm localization property, then $\phi$ extends to a bounded homomorphism
\[ \phi:B^p(\square G)\rightarrow\frac{\prod_{i=1}^\infty B^p(|G|)^{N_i}}{\bigoplus_{i=1}^\infty B^p(|G|)^{N_i}}. \]
\end{lemma}

\begin{proof}
Suppose $G$ has the $p$-operator norm localization property relative to $f$ with constant $c$. Let $T\in\CC^p[\square G]$, and suppose $T$ has propagation $r$. For each sufficiently large $i$, there exists a nonzero $\xi\in\ell^p(G,\ell^p)$ with $\mathrm{diam}(\supp(\xi))\leq f(r)$ and $||T^{(i)}|| ||\xi||\geq||\widetilde{T^{(i)}}\xi||\geq c||\widetilde{T^{(i)}}|| ||\xi||$. Hence $||T||\geq||T^{(i)}||\geq c||\widetilde{T^{(i)}}||$ for all such $i$, so $||\phi(T)||\leq\limsup_i||\widetilde{T^{(i)}}||\leq\frac{1}{c}||T||$.
\end{proof}

\begin{definition}
Let $X$ be a proper metric space, and let $Z$ be a countable dense subset of $X$ used to define $\CC^p[X]$.
An operator $T\in B^p(X)$ is said to be a ghost if for all $R,\varepsilon>0$, there exists a bounded set $B\subseteq X$ such that if $\xi\in\ell^p(Z,\ell^p)$ is of norm one and supported in the open ball of radius $R$ about some $x\notin B$, then $||T\xi||<\varepsilon$.
\end{definition}

The proof of the following lemma is exactly the same as that of \cite[Lemma 5.5]{WillettYuI}, and we include it for the convenience of the reader.

\begin{lemma}	\label{ghostlift}
Suppose that $G$ has the $p$-operator norm localization property. Let
\[ \phi:B^p(\square G)\rightarrow\frac{\prod_{i=1}^\infty B^p(|G|)^{N_i}}{\bigoplus_{i=1}^\infty B^p(|G|)^{N_i}} \]
be the homomorphism in Lemma \ref{phihom}. Then $\phi(T^G)=0$ for any ghost operator $T^G$.
\end{lemma}

\begin{proof}
Fix $\varepsilon>0$. Let $T^G$ be a ghost operator, and let $T\in\CC^p[\square G]$ have propagation $R$ and be such that $||T^G-T||<\varepsilon$. Let $\widetilde{T^{(i)}}$ be as in the definition of $\phi(T)$, and note that each $\widetilde{T^{(i)}}$ has propagation at most $R$. Suppose $G$ has the $p$-operator norm localization property relative to $f$ with constant $c$. Then for each $i$, there exists a nonzero $\widetilde{\xi_i}\in\ell^p(G,\ell^p)$ of norm one with support diameter at most $f(R)$ such that
\[ ||\widetilde{T^{(i)}}\widetilde{\xi_i}||\geq c||\widetilde{T^{(i)}}||. \]
On the other hand, for all sufficiently large $i$, there exists $\xi_i\in\ell^p(G/N_i,\ell^p)$ of norm one such that $||\widetilde{T^{(i)}}\widetilde{\xi_i}||=||T^{(i)}\xi_i||$. For such $i$, since $T^G$ is a ghost, we have
\[ c||\widetilde{T^{(i)}}||\leq||T^{(i)}\xi_i||\leq||T^G-T||+||T^G\xi_i||<2\varepsilon. \]
Hence
\[ ||\phi(T^G)||<\varepsilon||\phi||+||\phi(T)||\leq\varepsilon||\phi||+\limsup_i||\widetilde{T^{(i)}}||<\varepsilon||\phi||+\frac{2\varepsilon}{c}. \]
Since $\varepsilon$ is arbitrary, and $c$ is independent of $\varepsilon$, this completes the proof.
\end{proof}

\begin{remark}
One can consider the $\ell^p$ uniform Roe algebra $B^p_u(\square G)$, defining 
\[ \phi_u:\CC^p_u[\square G]\rightarrow\frac{\prod_i\CC^p_u[G]^{N_i}}{\bigoplus_i\CC^p_u[G]^{N_i}} \]
in a completely analogous manner using the same formula for the lifts of operators.
If $G$ has the $p$-operator norm localization property, then $\phi_u$ extends to a bounded homomorphism
\[ \phi_u:B^p_u(\square G)\rightarrow\frac{\prod_iB^p_u(|G|)^{N_i}}{\bigoplus_iB^p_u(|G|)^{N_i}} \]
by essentially the same proof as that of Lemma \ref{phihom}.
Moreover, if $e$ is a fixed rank one idempotent operator on $\ell^p$, then we have the following commutative diagram:
\[
\begin{CD}
\frac{\prod_iB^p_u(|G|)^{N_i}}{\bigoplus_iB^p_u(|G|)^{N_i}}		@>\prod_i(\cdot\otimes e)>>	\frac{\prod_iB^p(|G|)^{N_i}}{\bigoplus_iB^p(|G|)^{N_i}}			\\
@A\phi_u AA  @AA\phi A \\
B^p_u(\square G)				@>\cdot\otimes e>>			B^p(\square G)	
\end{CD}
\]
\end{remark}

\section{Kazhdan projections in the $\ell^p$ Roe algebra}
At this point an interlude is necessary in order to introduce and discuss Kazhdan projections in the setting of $\ell^p$ spaces. Our description follows that of \cite{drutu-nowak}.
A representation $\pi$ of $G$ on a Banach space $E$ induces a representation of $\CC G$ on $E$ by the formula
$$\pi(f)=\sum_{g\in G} f(g)\pi_g,$$
for every $f\in \CC G$. For a faithful family $\mathcal{F}$ of representations of $G$ on $\ell^p$, consider the following norm on $\CC G$,
$$\Vert f\Vert_{\mathcal{F},p} = \sup \setd{\Vert \pi(f)\Vert_{\ell^p}}{\pi \in\mathcal{F} }.$$
The completion of $\CC G$ in this norm will be denoted $B^p_{\mathcal{F}}(G)$.

Recall that given an isometric representation $\pi$ of a locally compact group $G$ on a reflexive Banach space $E$, the dual 
space $E^*$ is naturally equipped with the representation $\opi_g=(\pi_g^{-1})^*$.
We have a canonical decomposition of $\pi$ into 
the trivial representation and its complement,
$$E=E^{\pi} \oplus_{\pi} E_{\pi},$$
where $E^{\pi}$ is the subspace of invariant vectors of $\pi$, and $E_{\pi}=\operatorname{Ann}((E^*)^{\opi})$ (cf. \cite[Proposition 2.6 and Example 2.29]{bader2007}).

\begin{definition}
A Kazhdan projection $p\in B^p_{\mathcal{F}}(G)$ is an idempotent such that $\pi(p)\in B(\ell^p)$ is the projection onto $(\ell^p)^\pi$ along $(\ell^p)_{\pi}$ for every $\pi\in \mathcal{F}$.
\end{definition}

Given a finite graph $\Gamma=(V,E)$, the edge boundary $\partial A$ of a subset $A\subseteq V$ is defined to be 
the set of those edges $E$ that have exactly one vertex in $A$. The Cheeger constant of $\Gamma$ is then defined to be
$$h(\Gamma)=\inf_{A\subseteq V} \dfrac{\#\partial A}{\min \{ \#A, \#V\setminus A\}}.$$
Recall that a sequence of finite graphs $\{ \Gamma_n\}$ is a sequence of expanders if 
$$\inf_n h(\Gamma_n)>0.$$
See e.g. \cite{Hooryetal} for an overview.

Let $G$ be a residually finite group and let $\{N_i\}_{i\in \NN}$ be a family of finite index, normal subgroups. 
Consider the box space 
$X=\bigsqcup_i G/ N_i$, as before, and let $\pi_i$ be the quasi-regular representation of $G$ on $\ell^p(G/N_i)$. 
Fix $\mathcal{F}$ to denote the family of representations $\{\pi_i\}_{i\in\mathbb{N}}$.
We say that $G$ has property $\tau$ with respect to the sequence $\{N_i\}$ if the Cayley graphs $G/N_i$ form a family of expanders,
see \cite{lubotzky}.
The following is a special case of \cite[Theorem 1.2]{drutu-nowak}.
\begin{theorem}\cite{drutu-nowak}
Let $G$ be a finitely generated group with property $\tau$ with respect to the family $\left\{N_i\right\}$. 
Then for every $1<p<\infty$, there exists a Kazhdan idempotent in $B^p_{\mathcal{F}}(G)$.
\end{theorem}

Under the assumptions of the above theorem we have 
\begin{theorem}
There exists a non-compact ghost idempotent $Q=Q^2$ in $B^p(X)$.
\end{theorem}
\begin{proof}[Sketch of proof]
Consider the projection $q_i=\dfrac{1}{[G:N_i]} M_i$, where $M_i$ is a square matrix indexed by the elements $G/N_i$ with all entries equal to 1. Then
$q=\bigoplus  q_i$ belongs to $B^p_u(X)$. Indeed, by the construction in \cite{drutu-nowak} one can choose a finitely supported probability measure $\mu$ on $G$ such
that
$$\Vert \mu^n - q_i\Vert \to 0$$
uniformly in $i$. As $\mu^n$ are finite propagation operators, it suffices to take $Q$ to be the matrix defined by $Q(x,y)=q(x,y)P$, where $P$ is some rank one projection on $\ell^p$.
\end{proof}

\section{Main result}

After recording a few more ingredients, we shall be ready for the proof of the main result.

For any compact metric space $Y$, the algebra $B^p(Y)$ is isomorphic to the algebra of compact operators $K(\ell^p)$. 
In this case, we have \[ K_0(B^p(Y))\cong K_0(K(\ell^p))\cong \varinjlim K_0(M_n)\cong \mathbb{Z}.\]
The isomorphism $K_0(K(\ell^p))\cong\mathbb{Z}$ can also be seen to be induced by the canonical densely defined trace (see \cite[Subsection 1.7.11]{Pal}) on $K(\ell^p)$.

Let $G$ be a residually finite group, let $N_1\supseteq N_2\supseteq\cdots$ be a sequence of normal subgroups of finite index such that $\bigcap_iN_i=\{e\}$, and let $\square G=\bigsqcup_i G/N_i$ be the box space.
Note that if $n<R$, then \[P_R(\square G)=P_R(\bigsqcup_{i=1}^{n-1}G/N_i)\sqcup\bigsqcup_{i\geq n}P_R(G/N_i).\]
Each $N_i$ acts properly on $P_R(G)$. Moreover, if $B(e,r)\cap N_i=\{e\}$, then the action of $N_i$ on $P_R(G)$ is free, and $\pi:P_R(G)\rightarrow P_R(G)/N_i$ is a covering map (cf. \cite[Lemma 4.2]{OYmax}). Since $N_i$ has finite index in $G$, this covering is cocompact.
Moreover, we have an isomorphism
\[K_0(B^p_L(P_R(\square G))) \cong K_0(B^p_L(P_R(\bigsqcup_{i=1}^{n-1}G/N_i)))\oplus\prod_{i\geq n}K_0(B^p_L(P_R(G/N_i))),\]
which is the cluster axiom for $K$-homology saying that the $K$-homology of a countable disjoint union of spaces is isomorphic to the direct product of the $K$-homology of the individual spaces (cf. \cite[Section 7.4]{HR}).

Since $d(G/N_i,G/N_j)\geq i+j$ if $i\neq j$, 
we also have a homomorphism \[ d:B^p(\square G)\rightarrow\frac{\prod_i B^p(|G/N_i|)}{\bigoplus_i B^p(|G/N_i|)}\cong\frac{\prod_i K(\ell^p(G/N_i,\ell^p))}{\bigoplus_i K(\ell^p(G/N_i,\ell^p))}, \] 
which induces a homomorphism
\[ d_*:K_0(B^p(\square G))\rightarrow K_0\left(\frac{\prod_i K(\ell^p(G/N_i,\ell^p))}{\bigoplus_i K(\ell^p(G/N_i,\ell^p))}\right)\rightarrow \frac{\prod_i K_0(K(\ell^p(G/N_i,\ell^p)))}{\bigoplus_i K_0(K(\ell^p(G/N_i,\ell^p)))} \cong\frac{\prod_i \ZZ}{\bigoplus_i \ZZ}. \]
This map can be used to detect non-zero classes in $K_0(B^p(\square G))$.
Here, for a sequence $(A_i)_{i\in\mathbb{N}}$ of Banach algebras, we have denoted by $\prod_iA_i$ the collection of bounded sequences $\{(a_i)_{i\in\mathbb{N}}:a_i\in A_i,\sup_i||a_i||<\infty\}$, and we have denoted by $\bigoplus_iA_i$ the collection of sequences $(a_i)\in\prod_iA_i$ such that $\lim_{i\rightarrow\infty}||a_i||=0$.

The following proposition is a restatement of \cite[Proposition 2.7]{GWY} after identifying equivariant $K$-homology with the $K$-theory of equivariant localization algebras.

\begin{proposition} \label{KhomInj}
If the classifying space for proper $\Gamma$-actions has finite homotopy type, i.e., there is a model $Z$ of a locally finite CW complex with universal proper $\Gamma$-action such that $Z/\Gamma$ is a compact CW complex, then for any $r>0$, there is $R>0$ such that the following is true: for any two elements $[x],[y]\in K_*(B^p_L(P_r(\Gamma))^{\Gamma'})$, where $\Gamma'$ is a subgroup of $\Gamma$ with finite index, if $[x]=[y]$ in $\lim_{r\rightarrow\infty}K_*(B^p_L(P_r(\Gamma))^{\Gamma'})$, then $[x]=[y]$ in $K_*(B^p_L(P_R(\Gamma))^{\Gamma'})$.
\end{proposition}

By \cite[Theorem 1]{Meintrup-Schick}, the proposition applies to hyperbolic groups.
Now we have all the necessary ingredients to prove our main result.

\begin{theorem}
Let $p\in(1,\infty)$. Let $G$ be a residually finite hyperbolic group. Let $N_1\supseteq N_2\supseteq\cdots$ be a sequence of normal subgroups of finite index such that $\bigcap_iN_i=\{e\}$. If $q\in B^p(\square G)$ is a non-compact ghost idempotent, then $[q]\in K_0(B^p(\square G))$ is not in the image of the $\ell^p$ coarse Baum-Connes assembly map.

In particular, assuming that the box space $\square G=\bigsqcup_i G/N_i$ is an expander, i.e., that $G$ has property $\tau$ with respect to the family $\{N_i\}$, the class of the Kazhdan projection is not in the image of the $\ell^p$ coarse Baum-Connes assembly map. 
\end{theorem}

\begin{proof}
If $q$ is a non-compact ghost idempotent, then $d_*[q]\neq 0$ so $[q]\neq 0$ in $K_0(B^p(\square G))$.
On the other hand, since $q$ is a ghost operator, and hyperbolic groups have the $p$-operator norm localization property by Corollary \ref{pONLhyp}, we have $\phi(q)=0$ by Lemma \ref{ghostlift} so $\phi_*[q]=0$ in $\frac{\prod_iK_0(B^p(|G|)^{N_i}}{\bigoplus_iK_0(B^p(|G|)^{N_i}}$.

Consider the following diagram for fixed $R$ and $n<R$, in which the horizontal arrows are given by the respective assembly maps and the left vertical arrow is given by the induction isomorphisms from Proposition \ref{ind}.
\[
\begin{CD}
0\oplus\prod_{i\geq n}K_0(B^p_L(P_R(G))^{N_i})		@>>>	\frac{\prod_iK_0(B^p(|G|)^{N_i})}{\bigoplus_iK_0(B^p(|G|)^{N_i})}			\\
@A0\oplus\prod_{i\geq n}ind_iAA  @AA\phi_*A \\
K_0(B^p_L(P_R(\bigsqcup_{i=1}^{n-1}G/N_i)))\oplus\prod_{i\geq n}K_0(B^p_L(P_R(G/N_i)))				@>>>			K_0(B^p(\square G))	
\end{CD}
\]
It follows from the definitions of the maps that the diagram commutes.

Suppose $[q]$ is in the image of the $\ell^p$ coarse Baum-Connes assembly map.
Then there exist $r>0$ and $x\in K_0(B^p_L(P_r(\square G)))$ that maps to $[q]$.

Choose $(y_i)\in\prod_i K_0(B^p(|G|)^{N_i})$ that lifts $\phi_*[q]$. Then there exists $M_r$ such that $y_i=0$ in $K_0(B^p(|G|)^{N_i})$ for all $i\geq M_r$.
Each $N_i$ is hyperbolic so the $\ell^p$ Baum-Connes assembly map is injective for each $N_i$ by Corollary \ref{LpBC}.
Thus by Proposition \ref{KhomInj}, there exists $R>0$ independent of the subgroups $N_i$, and there exists $M_R\geq M_r$ such that $ind_i(x_i)=0$ for all $i\geq M_R$, and thus $x_i=0$ for all $i\geq M_R$.
But this contradicts $[q]\neq 0$.

Finally, under the expander assumption, the Kazhdan projection is a non-compact ghost idempotent in $B^p(\square G)$ so its class is not in the image of the $\ell^p$ coarse Baum-Connes assembly map.
\end{proof}

\section{Remarks and open questions}

In this final section, we list a few questions that we do not have the answer to and that may be of interest to the reader.

In the results above it was necessary to assume $1<p<\infty$. The main reason for this assumption is the construction of Kazhdan projections and the associated
ghost projection in $B^P(X)$. Indeed, the techniques used here and originally in \cite{drutu-nowak} require uniform convexity of the underlying Banach space.
Therefore the following question is natural in this context.
\begin{question}
What happens when $p=1$? (How to construct Kazhdan projections?)
\end{question}
The above situation bears certain resemblance to the case of the Bost conjecture, in which the right hand side of the Baum-Connes conjecture, namely 
the $K$-theory of the reduced group $C^*$-algebra $C^*_r(G)$ is replaced with the $K$-theory of the Banach algebra $\ell^1(G)$.

It also seems that the argument used here would extend also to the case of $\ell^p(Z,E)$, where $p>1$ and $E$ is a uniformly convex Banach space, or a Banach space of nontrivial type.
In this case also $\ell^p(Z,E)$ is uniformly convex, or has nontrivial type, respectively. Lafforgue \cite{lafforgue} and Liao \cite{liao} showed that in such cases there exist Kazhdan projections
and it would therefore be possible to construct the associated ghost projection. It is natural to state
\begin{question}
What happens if we consider operators on $\ell^p(Z,E)$ for Banach spaces $E$ other than $\ell^p$ in the definition of the Roe algebra?
\end{question}
As noted earlier, even $E=L^p[0,1]$ results in an algebra that is not isomorphic to the one we have used in this paper, although its $K$-theory may be the same.

Our formulation of the $\ell^p$ coarse Baum-Connes assembly map is based on a straightforward modification of a model of the original coarse Baum-Connes assembly map from \cite{LocAlgBC}.
One can check that for each $p$ and for a metric space $X$ with bounded geometry, the functors $X\mapsto\lim_{R\rightarrow\infty}K_n(B^p_L(P_R(X)))$ form a coarse homology theory in the sense of \cite[Definition 2.3]{HR95}.
What is of interest, and which goes back to one of the original motivations for studying $L^p$ analogs of Baum-Connes type assembly maps, is to identify the left-hand side of these assembly maps.
The following question has been answered affirmatively for finite-dimensional simplicial complexes in \cite[Proposition 5.20]{ZhangZhou}.

\begin{question}
Is the $K$-theory of $\ell^p$ localization algebras associated to bounded geometry metric spaces independent of $p$?
\end{question}

In this paper, we considered the problem of whether the $\ell^p$ coarse assembly map is surjective but we have not considered injectivity, whereas it was shown in \cite{WillettYuI} for the $p=2$ case that the coarse assembly map is injective for spaces of graphs with large girth.
\begin{question}
Is the $\ell^p$ coarse Baum-Connes assembly map injective for the expanders considered in this paper?
\end{question}

Finally, we would like to comment about the $p$-operator norm localization property.
Combining Proposition \ref{MSP} with \cite[Theorem 5.1]{Sako}, we know that a bounded geometry metric space with the 2-operator norm localization property will have the $p$-operator norm localization property for all $p\in[1,\infty)$. We do not know whether the converse holds, but we thank the referee for suggesting a possible approach to this problem, namely by considering uniform local amenability along the lines of \cite{ULA} and \cite{ULA2}.

\appendix
\section{Alternative description of the assembly map}

In this appendix, we provide an alternative description of the $\ell^p$ coarse Baum-Connes assembly map, following the approach in the $p=2$ case from \cite{QR} and \cite[Appendix B]{GuWY}.
We do this by considering a larger algebra generated by pseudolocal operators, and we also show that classes in the image of the $\ell^p$ coarse Baum-Connes assembly map can be represented by elements with finite propagation (at least for $K_0$).

For a proper metric space $X$, and a countable dense subset $Z\subseteq X$, the Banach space $\ell^p(Z,\ell^p)$ is equipped with a natural multiplication action of $C_0(X)$ by restriction to $Z$.

\begin{definition} \label{def:dpl}
Let $X$ be a proper metric space, and fix a countable dense subset $Z\subseteq X$. A bounded operator $T$ on $\ell^p(Z,\ell^p)$ is said to be pseudolocal if the commutator $[f,T]$ is a compact operator for every $f\in C_0(X)$.

Let $\DD^p[X]$ be the algebra of all finite propagation, pseudolocal operators on $\ell^p(Z,\ell^p)$,
and define $D^p(X)$ to be the closure of $\DD^p(X)$ in $B(\ell^p(Z),\ell^p)$.

Let $\DD^p_L[X]$ be the algebra of all bounded, uniformly continuous functions $f:[0,\infty)\rightarrow\DD^p[X]$ such that $\prop(f(t))\rightarrow 0$ as $t\rightarrow\infty$,
and define $D^p_L(X)$ to be the completion of $\DD^p_L[X]$ under the supremum norm.
\end{definition}

Note that $B^p(X)$ and $B^p_L(X)$ are closed ideals in $D^p(X)$ and $D^p_L(X)$ respectively for a proper metric space $X$ with bounded geometry.

If $T$ is pseudolocal, then $fTg$ is a compact operator for every pair of functions $f,g\in C_0(X)$ with disjoint supports. This is because $fTg=f[T,g]$.
Conversely, we have the following version of Kasparov's lemma (cf. \cite[Proposition 3.4]{Kasp}) whose proof follows closely that of \cite[Lemma 5.4.7]{HR}:
\begin{lemma} \label{pKaspLem}
Let $X$ be a proper metric space, let $Z$ be a countable dense subset of $X$, and let $T\in B(\ell^p(Z,\ell^p))$. Suppose that $fTg$ is a compact operator for every pair of functions $f,g\in C_0(X)$ with disjoint compact supports. Then $T$ is pseudolocal.
\end{lemma}

For the rest of this subsection, we will assume that $X$ is a proper metric space with bounded geometry.

\begin{proposition} \label{PropDLtrivialK}
The algebra $D^p_L(X)$ has trivial $K$-theory. Hence the boundary map $\partial:K_{*+1}(D^p_L(X)/B^p_L(X))\rightarrow K_*(B^p_L(X))$ is an isomorphism.
\end{proposition}

The proof of the proposition is the same as that of \cite[Proposition 3.5]{QR}.

%

Now we want to show that classes in $K_1(D^p_L(X)/B^p_L(X))$ have finite propagation representatives whose inverses also have finite propagation. From this, we see that classes in $K_0(B^p_L(X))$ have finite propagation representatives.
To simplify notation, given a Banach algebra $A$, we will write $\mathfrak{T}A$ to denote the algebra of all bounded, uniformly continuous functions from $[0,\infty)$ to $A$, equipped with the supremum norm.

Note that $\mathfrak{T}D^p(X)$ contains $\mathfrak{T}B^p(X)$ as a closed ideal, while $D^p_L(X)$ contains $B^p_L(X)$ as a closed ideal. Thus we have the following commutative diagram of short exact sequences, where the upper vertical maps are induced by evaluation at zero, while the lower vertical maps are induced by forgetting the condition that propagation tends to zero.
\[\minCDarrowwidth20pt
\begin{CD}
0 @>>>	B^p(X)	@>>>	D^p(X)	@>>>	D^p(X)/B^p(X)	@>>> 0	\\
@. @AAA @AAA @AAA \\
0 @>>>	\mathfrak{T}B^p(X)	@>>>	\mathfrak{T}D^p(X)	@>>>	\mathfrak{T}D^p(X)/\mathfrak{T}B^p(X)	@>>> 0	\\
@. @AAA @AAA @AAA \\
0 @>>>	B^p_L(X)	@>>>	D^p_L(X)	@>>>	D^p_L(X)/B^p_L(X)	@>>> 0	\\
\end{CD}
\]

The proof of the next lemma involves truncating operators using a partition of unity in a way adapted to the $\ell^p$ norm and the dual $\ell^q$ norm (cf. \cite[Lemma 6.3]{WillettSpakula}) for $p\in(1,\infty)$.
Other than this little modification, the proof is the same as that of \cite[Lemma 2.2]{QR}.

\begin{lemma} \label{TruncLem}
There is a contractive linear map $\Phi:\mathfrak{T}D^p(X)\rightarrow D^p_L(X)$ such that
\begin{enumerate}
\item $\Phi(a)(t)$ depends only on $a(t)$,
\item $\prop(\Phi(a)(t))\leq\prop(a(t))$ for all $a\in \mathfrak{T}D^p(X)$,
\item $\Phi(a)(t)-a(t)\in B^p(X)$ for all $a\in \mathfrak{T}D^p(X)$.
\end{enumerate}
Moreover, $\Phi$ induces a homomorphism \[\widetilde{\Phi}:\mathfrak{T}D^p(X)/\mathfrak{T}B^p(X)\rightarrow D^p_L(X)/B^p_L(X).\]
\end{lemma}

\begin{remark}
Going through the proof of the lemma, one sees that given any $k\geq 1$, one can actually arrange for $\prop(\Phi(a)(t))\leq 1/k$ for all $t\in[0,\infty)$.
\end{remark}

The next proposition can be proved in the same way as \cite[Proposition 2.3]{QR}.

\begin{proposition} \label{PropIncIsom}
The homomorphism \[\iota:D^p_L(X)/B^p_L(X)\rightarrow \mathfrak{T}D^p(X)/\mathfrak{T}B^p(X)\] induced by inclusion is an isomorphism (with inverse $\widetilde{\Phi}$).
\end{proposition}

%

\begin{corollary} \label{CorFinProp}
Every class in $D^p_L(X)/B^p_L(X)$ can be represented by an element in $D^p_L(X)$ with finite (arbitrarily small) propagation. For every class $[u]\in K_1(D^p_L(X)/B^p_L(X))$ with $u$ invertible in $M_n(D^p_L(X)/B^p_L(X))$ for some $n$, there is a lift $x\in M_n(D^p_L(X))$ of $u$ with finite propagation, and there is a lift $y\in M_n(D^p_L(X))$ of $u^{-1}$ with finite propagation.
\end{corollary}

From an explicit formula for the $K$-theory index map (see \cite[Definition 1.46 and Equation (1.43)]{CMR}), we get the following:

\begin{corollary} \label{finproprep}
Every class in $K_0(B^p_L(X))$ can be represented as a formal difference $[p]-[q]$, where $p$ and $q$ are idempotents with finite propagation in matrix algebras over the unitization $\widetilde{B^p_L(X)}$.

Hence, every class in the image of the $\ell^p$ coarse Baum-Connes assembly map can be represented as a formal difference $[p']-[q']$, where $p'$ and $q'$ are idempotents with finite propagation in matrix algebras over the unitization $\widetilde{B^p(X)}$.
\end{corollary}

\begin{proposition} \label{PropEvalIsom}
The homomorphisms 
\begin{align*}
\mathfrak{T}B^p(X)&\rightarrow B^p(X), \\
\mathfrak{T}D^p(X)&\rightarrow D^p(X), \\
\mathfrak{T}D^p(X)/\mathfrak{T}B^p(X)&\rightarrow D^p(X)/B^p(X),
\end{align*} 
induced by evaluation at zero induce isomorphisms on $K$-theory.
\end{proposition}

Using the six-term exact sequence and the five lemma, it suffices to show that the first two homomorphisms induce isomorphisms on $K$-theory. This can be done in essentially the same way as in the proof of \cite[Proposition 3.6]{QR}.

%

Combining the isomorphisms in Propositions \ref{PropDLtrivialK}, \ref{PropIncIsom}, and \ref{PropEvalIsom}, we get the following commutative diagram:
\[
\begin{CD}
K_{*+1}(D^p(X)/B^p(X))	@>\partial>>	K_*(B^p(X))		\\
@A\cong Ae_0A @AAe_0A \\
K_{*+1}(D^p_L(X)/B^p_L(X))	@>\cong>\partial>	K_*(B^p_L(X))
\end{CD}
\]
In particular, $\partial:K_{*+1}(D^p(X)/B^p(X))\rightarrow K_*(B^p(X))$ is an isomorphism if and only if $e_0:K_*(B^p_L(X))\rightarrow K_*(B^p(X))$ is an isomorphism.

Applying this fact to each Rips complex and taking limits, we get

\begin{theorem} \label{ThmAltAssembly}
The $\ell^p$ coarse Baum-Connes conjecture holds for $X$ if and only if the limit (as $R\rightarrow\infty$) of the compositions
\[ K_{*+1}(D^p(P_R(X))/B^p(P_R(X)))\stackrel{\partial}{\rightarrow} K_*(B^p(P_R(X)))\stackrel{i_R}{\rightarrow} K_*(B^p(X)) \]
is an isomorphism.
\end{theorem}

We note that the results above can be adapted to the equivariant case (cf. \cite[Appendix B]{GuWY} when $p=2$).

Finally, we note that failure of the $\ell^p$ coarse Baum-Connes conjecture can be detected by certain obstruction groups (cf. \cite{CW,Yu98,Yu00} when $p=2$).

\begin{definition}
For a proper metric space $X$ with bounded geometry, define 
\begin{align*}
B^p_{L,0}(X) &= \{f\in B^p_L(X):f(0)=0\},\\ 
D^p_{L,0}(X) &= \{f\in D^p_L(X):f(0)=0\}.
\end{align*}
\end{definition}

Note that $B^p_{L,0}(X)$ is a closed ideal in $D^p_{L,0}(X)$. Moreover, we have the following commutative diagram with exact rows and columns (cf. \cite[Section 5]{CW}):
\[\minCDarrowwidth20pt
\begin{CD}
@.	0	@.	0	@.	0	@.		\\
@. @VVV @VVV @VVV @. \\
0 @>>> B^p_{L,0}(X)	@>>>	D^p_{L,0}(X)	@>>>	D^p_{L,0}(X)/B^p_{L,0}(X) @>>> 0 \\
@. @VVV @VVV @VVV @. \\
0 @>>> B^p_L(X)	@>>>	D^p_L(X)	@>>>	D^p_L(X)/B^p_L(X) @>>> 0 \\
@. @VVe_0V @VVe_0V @VVe_0V @. \\
0 @>>> B^p(X)	@>>>	D^p(X)	@>>>	D^p(X)/B^p(X) @>>> 0 \\
@. @VVV @VVV @VVV @. \\
@.	0	@.	0	@.	0	@.
\end{CD}
\]
Applying the $K$-theory functor and applying the earlier results, we then get the following commutative diagram with exact rows and columns:
\[\minCDarrowwidth5pt
\begin{CD}
@.	\vdots	@.	\vdots	@.	\vdots	@.	\vdots	\\
@. @VVV @VVV @VVV @VVV @. \\
\cdots @>>> K_{*+1}(\frac{D^p_{L,0}(X)}{B^p_{L,0}(X)}) @>>> K_*(B^p_{L,0}(X))	@>>>	K_*(D^p_{L,0}(X))	@>>>	K_*(\frac{D^p_{L,0}(X)}{B^p_{L,0}(X)}) @>>> \cdots \\
@. @VVV @VVV @VVV @VVV @. \\
\cdots @>>> K_{*+1}(\frac{D^p_L(X)}{B^p_L(X)}) @>\cong>> K_*(B^p_L(X))	@>>>	0	@>>>	K_*(\frac{D^p_L(X)}{B^p_L(X)}) @>>> \cdots \\
@. @V\cong VV @VVV @VVV @V\cong VV @. \\
\cdots @>>> K_{*+1}(\frac{D^p(X)}{B^p(X)}) @>>> K_*(B^p(X))	@>>>	K_*(D^p(X))	@>>>	K_*(\frac{D^p(X)}{B^p(X)}) @>>> \cdots \\
@. @VVV @VVV @VVV @VVV @. \\
@.	\vdots	@.	\vdots	@.	\vdots	@.	\vdots
\end{CD}
\]
It follows that $K_*(D^p_{L,0}(X)/B^p_{L,0}(X))=0$ so \[K_*(B^p_{L,0}(X))\cong K_*(D^p_{L,0}(X))\cong K_{*+1}(D^p(X)).\]
\begin{corollary}
The $\ell^p$ coarse Baum-Connes conjecture holds for $X$ if and only if one of the following vanishes:
\begin{itemize} 
\item $\lim_{R\rightarrow\infty}K_*(B^p_{L,0}(P_R(X)))$, 
\item $\lim_{R\rightarrow\infty}K_*(D^p_{L,0}(P_R(X)))$, 
\item $\lim_{R\rightarrow\infty}K_*(D^p(P_R(X)))$. 
\end{itemize}
\end{corollary}
In other words, any of the above can serve as an obstruction group for the $\ell^p$ coarse Baum-Connes assembly map.

\bibliographystyle{amsxport}
\bibliography{bib}

\end{document}